\documentclass[11pt,a4paper]{article}

\usepackage[a4paper, total={6in, 8in}]{geometry}
\usepackage[english]{babel}
\usepackage{amsmath}
\usepackage{theorem}
\numberwithin{equation}{section}
\let\theoremstyle\undefined
\usepackage{amsthm}
\usepackage{hyperref}
\usepackage{amsfonts}
\usepackage{amssymb}
\usepackage{mathtools}
\usepackage{bbm, mathrsfs, dsfont, cmll}
\usepackage{siunitx}
\usepackage{amsfonts}
\usepackage{amssymb}
\usepackage{graphicx}
\usepackage{epstopdf}
\usepackage[normalem]{ulem} 
\usepackage{algorithmic}

\usepackage[caption=false]{subfig}

\ifpdf
  \DeclareGraphicsExtensions{.eps,.pdf,.png,.jpg}
\else
  \DeclareGraphicsExtensions{.eps}
\fi

\usepackage{booktabs}
\usepackage{multirow}
\usepackage{colortbl}

\usepackage{url}
\usepackage{enumerate}
\usepackage[neverdecrease]{paralist}
\usepackage{etoolbox}
\usepackage{verbatim}
\usepackage{comment}
\usepackage{framed, color}
\usepackage{pifont, textcomp}
\usepackage{bbding}
\usepackage{scalefnt}
\usepackage[outline]{contour}
\usepackage[framemethod=tikz]{mdframed}

\usepackage{tikz}
\usepackage{pgfplots}
\pgfplotsset{compat=1.18}
\usetikzlibrary{
  calc, fadings, matrix, positioning, spy,
  decorations.pathreplacing, decorations.text, decorations.pathmorphing,
  shapes.geometric, arrows.meta, svg.path, calligraphy, 3d, plotmarks
}

\usepackage[noabbrev]{cleveref}
\crefname{equation}{equation}{equations}

\usepackage{enumitem}
\setlist[enumerate]{leftmargin=.5in}
\setlist[itemize]{leftmargin=.5in}

\newtheorem{theorem}{Theorem}[]

\newtheorem{lemma}[theorem]{Lemma}

\newcounter{MotEx}

{
  \theoremstyle{plain}
  \theoremheaderfont{\normalfont}
  \theorembodyfont{\normalfont\itshape}
  
}

\def\<{\langle}
\def\>{\rangle}

\DeclareFontFamily{U}{mathx}{}
\DeclareFontShape{U}{mathx}{m}{n}{<-> mathx10}{}
\DeclareSymbolFont{mathx}{U}{mathx}{m}{n}
\DeclareMathAccent{\widehat}{0}{mathx}{"70}
\DeclareMathAccent{\widecheck}{0}{mathx}{"71}

\DeclareMathOperator*{\argmin}{argmin}

\DeclareMathOperator{\prox}{prox}

\usepackage{amsmath,amssymb,amsthm}
\usepackage{algorithm}

\usepackage{color}
\def\prox{{\mathrm{prox}}}
\makeatletter

\gdef\listctr{list\romannumeral\the\@listdepth}\expandafter

\makeatother

\newenvironment{AlgorithmSteps}[1][1]{%
\begin{list}{\csname label\listctr\endcsname}{%
\usecounter{\listctr}

\settowidth{\labelwidth}{\textsc{Step\ #1.}}%
\setlength{\leftmargin}{\labelwidth}\addtolength{\leftmargin}{\labelsep}}}%
{\end{list}}

\newcounter{cnt}

\newcommand{\email}[1]{\protect\href{mailto:#1}{#1}}

\title{A Line--Search--Based Stochastic Gradient Method for 3D Computed Tomography}

\author{Tatiana A. Bubba\thanks{Department of Mathematics and Computer Science of the University of Ferrara, Ferrara, Italy \\(\email{tatiana.bubba@unife.it}).}
\and Elena Morotti\thanks{Department of Political and Social Sciences of the University of Bologna, Bologna, Italy \\  (\email{elena.morotti4@unibo.it}).}
\and Federica Porta\thanks{Department of Physics, Informatics and Mathematics of the University of Modena and Reggio Emilia, Modena, Italy (\email{federica.porta@unimore.it}).}
\and Valeria Ruggiero\thanks{Department of Mathematics and Computer Science of the University of Ferrara, Ferrara, Italy.} 
\and Ilaria Trombini\thanks{Department of Mathematics and Computer Science of the University of Ferrara, Ferrara, Italy (\email{ilaria.trombini@unife.it}).}
}

\begin{document}

\maketitle

\begin{abstract}
We introduce FB--LISA, a forward--backward (FB) generalization of a recently proposed 
line--search--based stochastic gradient algorithm to address the  imaging problem of volumetric reconstruction in Computed Tomography,  a substantially high demanding problem, which involves orders of magnitude of data, a high computational burden for forward and backprojection, and memory requirements that push current GPU architectures to their limits. 
Our formulation employs stochastic mini--batches composed of full 2D projections, preserving the physical structure of the acquisition process while enabling significant speed--ups during early iterations.
The resulting method demonstrates how concepts traditionally associated with deep learning can be repurposed to accelerate large--scale inverse problems, without relying on training data or learned priors.
\end{abstract}

\noindent\textsc{Keywords}: Forward--backward methods, Image reconstruction, Optimization models,  \\ Stochastic algorithms.

\section{Introduction}
Stochastic optimization is a cornerstone of modern machine learning, enabling efficient training of large--scale models through updates based on small data subsets. By avoiding full gradient computations, stochastic methods provide a fast and memory--efficient alternative, often reaching good solutions well before theoretical convergence. This trade--off makes them appealing for large--scale optimization problems beyond learning.
Recently, stochastic optimization has gained attention in inverse problems (see, e.g., \cite{Chambolle-etal-2018, Ehrhardt-etal-2019, Tang-etal-2019, Tang-etal-2020, Chambolle-etal-2024, lazzaretti2023stochastic, Ehrhardt-etal-2025,papoutsellis2026modular}). In particular, stochastic gradient methods have shown promising results in imaging applications, where large--scale operators make full updates impractical. However, most existing works focus on small--scale or 2D settings, which do not fully reflect the computational challenges of real--world imaging systems.
In this work, we address 3D X--ray computed tomography (CT)~\cite{kak2001principles}, where reconstruction involves large--scale data and repeated applications of expensive forward and backprojection operators. In this context, deterministic methods may be impractical due to time and memory constraints.
To this end, we introduce a forward--backward (FB) 
stochastic gradient method, 
called FB--LISA, specifically designed for addressing large--scale inverse problems   and built on Deep--LISA \cite{Franchini-etal-2024},  a stochastic algorithm originally proposed for smooth finite--sum optimization. 
FB--LISA inherits Deep–LISA’s adaptive hyperparameter strategy, combining line--search--based learning rate estimation with a predetermined increase in mini--batch size. This removes manual tuning and improves robustness for large--scale applications.
In contrast to Deep--LISA, FB--LISA is designed to handle non--smooth objective terms, enabling the use of regularizers and constraints common in imaging. We demonstrate its effectiveness on 3D CT measured data, showing that it outperforms several state--of--the--art reconstruction methods.

The remainder of this work is organised as follows. In Section~\ref{sec:method} we introduce FB--LISA, along with our main theoretical results,  and the mathematical formulation of the problem. Section~\ref{sec:NumericalExperiments} demonstrates the effectiveness of FB--LISA through numerical experiments on measured 3D CT data, comparing its performance with that of several analytical, deterministic and stochastic approaches.  Section~\ref{sec:conclusions} summarizes our findings and discusses future research directions. Finally, the proof of our main result is reported in the Appendix~\ref{sec:appendix}.

\section{Method}\label{sec:method}
In this section, we formalize the mathematical framework for
our proposed method, FB--LISA, within the context of 3D Computed Tomography. In particular, we state the main convergence result for FB--LISA, whose proof is later provided in Appendix~\ref{sec:appendix}.

\subsection{Problem and notation}\label{ssec:ct3D}
We consider the following CT reconstruction problem:
\begin{equation}
    \label{eq:problem}
    \argmin_{x\in\mathbb{R}^d}
    \frac{1}{2} \|Ax-b\|^2 + \mu \|x\|_1 + \iota_{\{ x\geq 0\}},
\end{equation}
where $b\in\mathbb{R}^n$ is the set of measured data, and $A\in\mathbb{R}^{n\times d}$ denotes the discretized forward operator modeling the cone--beam projection, described through $n = n_p \times n_\theta$, with $n_\theta$ denoting the number of angular projections and $n_p$ the number of detector cells.  
The unknown $x \in \mathbb{R}^d$ represents the 3D discretized volume of attenuation coefficients.
The data--fidelity term $\frac{1}{2}\|Ax-b\|^2$ corresponds to a least--squares formulation, typically arising under an additive Gaussian noise model. The function $\iota_{\{x \geq 0\}}$ is the indicator function of the nonnegative orthant, enforcing physical feasibility of the reconstruction. Finally, $\mu>0$ is a regularization parameter controlling the sparsity level induced by the $\ell_1$ norm.

The problem is convex but nonsmooth and falls within the class of composite optimization problems, for which first--order methods are well suited. Several such algorithms have been proposed for problems of the form \eqref{eq:problem} (see, e.g., \cite{Boyd-etal-2011,Chambolle-etal-2011,Combettes-etal-2005,Condat-2013,dossal}).
The growing scale of modern CT systems leads to large reconstruction problems, where full--gradient evaluations become computationally expensive \cite{cavicchioli2020gpu}. This motivates stochastic and incremental approaches, which reduce per--iteration cost by operating on data subsets \cite{Bottou-etal-2018}. In CT, this principle underlies methods such as SART \cite{andersen1984simultaneous} and ordered--subsets (OS) techniques \cite{Erdogan-etal-1999,Herman-etal-1993,Kim-etal-2015}, where the forward operator $A$ is partitioned into $n_\theta$ blocks $A_i$ corresponding to individual projection angles. The computational advantage arises from the significantly lower cost of evaluating $A_i x$ compared to $Ax$.

\subsection{FB--LISA for 3D CT}\label{ssec:LISA_for_CT}
To address 
the optimization problem in \eqref{eq:problem}, we introduce FB--LISA (cf.~Alghoritm \ref{alg:FBLISA}, where FB--LISA pseudocode is reported), starting from the original Deep--LISA method \cite[Algorithm 2]{Franchini-etal-2024}. 
This requires to generalize Deep--LISA to handle the presence of a non--smooth term in the objective function. Consider a general optimization problem of the form 
\begin{equation}\label{eq:gen_prob}
\argmin_{x\in\mathbb{R}^d} F(x)\equiv f(x) + \mathcal{R}(x)\equiv \frac{1}{N} \sum_{i=1}^N f_i(x) + \mathcal{R}(x),
\end{equation}
where $f_i:\mathbb{R}^d\rightarrow\mathbb{R}$, $i=1,\dots,N$, are differentiable functions with $L$--Lipschitz continuous gradient  and $\mathcal{R}:\mathbb{R}^d\rightarrow\mathbb{R}$ is a non--smooth and convex term. In this setting, the resulting 
update rule and  line--search inequality (see \textsc{Step 2c} of Algorithm \ref{alg:FBLISA}) take the form
$$
\bar{x}^{(k)} = \prox_{\alpha_k \mathcal{R}}(x^{(k)}-\alpha_k\nabla f_{\mathcal{N}_k}(x^{(k)}))
$$
and 
\begin{equation}\label{eq:line_search}
    f_{\mathcal{N}_k}(\bar{x}^{(k)}) \leq f_{\mathcal{N}_k}(x^{(k)})+ \nabla f_{\mathcal{N}_k}(x^{(k)})^T(\bar{x}^{(k)}-x^{(k)})
    +\frac{1}{2\alpha_k}\|\bar{x}^{(k)}-x^{(k)}\|^2, 
\end{equation}
where $\prox_{\alpha_k\mathcal{R}}(\cdot)$ represents the proximal operator related to the function $\alpha_k\mathcal{R}(\cdot)$, $\mathcal{N}_k\subset\{1,\dots,N\}$ is chosen uniformly at random and $f_{\mathcal{N}_k} = \frac{1}{|\mathcal{N}_k|}\sum_{i\in\mathcal{N}_k}f_i(x)$ is the corresponding sub--sampled objective function. 

We have the following result for FB--LISA  applied to \eqref{eq:gen_prob}. We stress that the FB--LISA method is similar to the approach in \cite{Franchini-etal-2023a}. However, the practical realization of the variance reduction, the theoretical framework and the convergence analysis (see Appendix~\ref{sec:appendix}) are different.

\begin{theorem}\label{thm:MainResult}
Given problem \eqref{eq:gen_prob}, suppose that $f_i:\mathbb{R}^d\rightarrow\mathbb{R}$, $i=1,\dots,N$, are differentiable functions with $L$--Lipschitz continuous gradient, $\mathcal{R}:\mathbb{R}^d\rightarrow\mathbb{R}$ is convex and $F\equiv f$+$\mathcal{R}$ is bounded below by $F_{inf}$. Moreover, assume that $\alpha_{max}\leq\frac{1}{2L}$ and \\ $\mathbb{E}_k[\|\nabla f({x^{(k)})-\nabla f_{\mathcal{N}_k}(x^{(k)})}\|^2]\leq\varepsilon_k$ where $\mathbb{E}_k[\cdot]$ denotes conditional expected value with respect to the $\sigma$--algebra generated by  $x^{(0)},\dots,x^{(k)}$ and $\{\varepsilon_k\}_{k\in\mathbb{N}}$ is a positive and summable sequence. Let $\{x^{(k)}\}_{k\in\mathbb{N}}$ be the sequence generated by the FB--LISA method. Then, almost surely, any limit point of the sequence $\{x^{(k)}\}_{k\in\mathbb{N}}$ is a stationary point for problem (2).
\end{theorem}

To specifically apply FB--LISA to the CT problem in \eqref{eq:problem}, 
similarly to \cite{Ehrhardt-etal-2025, Tang-etal-2020} for 2D tomography, the fidelity term in \eqref{eq:problem} can be easily expressed as a finite sum. 
Specifically,
\begin{equation}\label{eq:fidelity_term_finite_sum}
\frac{1}{2}\|Ax - b\|^2 = \frac{1}{2 n_\theta} \sum_{i=1}^{n_\theta} \|A_i x - b_i\|^2,
\end{equation}
where \( A_i \in \mathbb{R}^{n_p \times d} \), for \( i=1, \ldots, n_\theta \), represents the CT operator associated with the \(i\)--th angle and $b_i \in \mathbb{R}^{n_p}$ denotes the corresponding projection data. 
The normalization factor $2n_\theta$ ensures that $f(x)$ can be interpreted as an empirical mean over projection angles, which is consistent with stochastic sampling strategies.
Next, we define the sub--sampled objective function $f_{\mathcal{N}_k}$ and the corresponding mini--batch selection process required in \textsc{Step 2a} in Algorithm \ref{alg:FBLISA}.
In contrast to the mini--batch selection strategy proposed in \cite{Ehrhardt-etal-2025,Tang-etal-2020}, in FB--LISA the angle indices for the mini--batch $\mathcal{N}_k$ are randomly and uniformly drawn from the set of all possible angle indices $\{1,\dots,n_\theta \}$. Given a mini--batch $\mathcal{N}_k$ of size $N_t$ at epoch $t$, we define
\begin{equation}\label{eq:f_Nk_CT}
\begin{aligned}
f_{\mathcal{N}_k}(x) 
&= \frac{n_{\theta}}{2N_t}\sum_{i\in\mathcal{N}_k}\|A_{i}x-b_i\|^2 = 
\frac{n_{\theta}}{2N_t}\|A_{\mathcal{N}_k}x-b_{\mathcal{N}_k}\|^2,
\end{aligned}
\end{equation} 
where \( A_{\mathcal{N}_k} \) denotes the restricted forward operator corresponding to the angles indexed by \( \mathcal{N}_k \), and $b_{\mathcal{N}_k} \in \mathbb{R}^{N_t}$ collects the corresponding measurements. We remark that the mini--batch size remains fixed throughout an entire epoch, corresponding to the computational cost of computing a full gradient.
A further difference from the original Deep--LISA implementation lies in the mini--batch construction. In the CT setting, the mini--batch $\mathcal{N}_k$ is generated on--the--fly by uniformly sampling projection angles at each iteration, rather than using pre--defined batches via data loaders. This strategy better exploits the structure of the forward operator while avoiding preprocessing and storage of data subsets.

\begin{algorithm}
 \caption{FB--LISA algorithm}
 \label{alg:FBLISA}

Given $T>0$, $n_{max}>0$, $x^{(0)}\in\mathbb{R}^d$, $0<N_0<n$, $\alpha_0>0$, $\beta \in (0,1)$ and a positive sequence $\{\varepsilon_k\}_{k\in\mathbb{N}}$, $\sum_k \varepsilon_k < \infty$, $C>0$, $\hat{k}=0.$

 \vspace{0.3cm}

 \textsc{For} $t=1,2,\dots,T$
 \begin{AlgorithmSteps}[4]

 \item[1]
 Choose the size $N_t$ according to
 \[
 N_{t}=\min\left\{n_{max},\max\left\{
 \left\lceil \frac{C}{\varepsilon_{\hat{k}+\lceil n/N_{t-1}\rceil}} \right\rceil,
 N_0
 \right\}\right\}.
 \]

 \item[2]
 \textsc{For} $k=\hat{k},\dots,\hat{k}+\lceil n/N_t \rceil -1$

 \begin{AlgorithmSteps}[4]

 \item[2a]
 Choose a mini-batch $\mathcal{N}_k$ of size $N_t$.

 \item[2b] Compute $f_{\mathcal{N}_k}(x^{(k)})$, $\nabla f_{\mathcal{N}_k}(x^{(k)})$.

 \item[2c]
 Compute
 \[
 \bar{x}^{(k)} = \prox_{\alpha_k \mathcal{R}}(x^{(k)}-\alpha_k \nabla f_{\mathcal{N}_k}(x^{(k)})).
 \]

 \noindent\textsc{If}
 \[
 f_{\mathcal{N}_k}(\bar{x}^{(k)}) \leq f_{\mathcal{N}_k}(x^{(k)})+ \nabla f_{\mathcal{N}_k}(x^{(k)})^T(\bar{x}^{(k)}-x^{(k)})
    +\frac{1}{2\alpha_k}\|\bar{x}^{(k)}-x^{(k)}\|^2 \]

 \noindent\textsc{Then} go to \textsc{Step 2d}.

 \noindent\textsc{Else} set $\alpha_k \leftarrow \beta \alpha_k$ and go to \textsc{Step 2c}.

 \item[2d]
 Set
 \[
 x^{(k+1)} = \bar{x}^{(k)}
 \]
 and
 \[
 \alpha_{k+1} = \alpha_0.
\]

\end{AlgorithmSteps}

\textsc{End For}

\item[3]
$\hat{k} = k+1$.

\end{AlgorithmSteps}

\textsc{End For}

\end{algorithm}

\section{Numerical experiments}\label{sec:NumericalExperiments}
In this section, we compare the proposed  
FB--LISA approach with several deterministic and stochastic algorithms.
All experiments are implemented in Matlab\textsuperscript{\textregistered} R2025b and executed on a workstation equipped with an Intel\textsuperscript{\textregistered} Xeon\textsuperscript{\textregistered} Gold 5418Y CPU with 128 GB of system RAM, and an NVIDIA RTX A1000 GPU with 8 GB of GPU memory. 
The code will be freely available on GitHub upon acceptance.

\paragraph{Dataset}
We consider the 3D cone--beam computed tomography (CBCT) dataset of a walnut in its shell, publicly available on Zenodo~\cite{dataset}. It consists of the projection images from the scan of a walnut over multiple view angles. The projections are binned and organized so that the final volume is of dimensions $560{\times}560{\times}560$. The dataset provides realistic acquisition measurements that are well--suited for evaluating reconstruction methods in challenging scenarios such as sparse--view settings.

As the dataset does not include a reference reconstruction, we generate the ground truth volume $x_{\text{GT}}$ by applying the Feldkamp--Davis--Kress (FDK) algorithm \cite{Feldkamp-Davis-Kress-1984} to the complete set of available projection views, i.e., 720 in [0, 2$\pi$). FDK is a standard analytical reconstruction technique for cone--beam geometry, available in most commercial CT scanners. When all projections are used, it provides a high--quality approximation of the underlying object, which therefore can be adopted as a reference for quantitative evaluation.

\paragraph{Similarity measures}
Reconstruction quality is assessed by comparing each reconstructed volume with the reference ground truth $x_{\text{GT}}$ using multiple quantitative metrics. Specifically, we consider the relative error (RE),  the peak signal--to--noise ratio (PSNR) and the structured similarity index (SSIM) \cite{ssim} provided by Matlab. Finally, we also report the Haar wavelet--based perceptual similarity index (HaarPSI), proposed in \cite{haar}. 
All metrics are computed over the entire 3D reconstructed volume, rather than on individual slices. 

\paragraph{Compared Methods}
We compare the proposed FB--LISA method, whose reconstruction will be hereafter denoted by $x_{\text{FB--LISA}}$, against the deterministic and stochastic algorithms listed below. 
For FB--LISA we adopt the following hyperparameter configuration (the notation is consistent with that of~\cite{Franchini-etal-2024}): $N_0 = 8$, $\alpha_0 = 10^{-3}$, $M = 1$ and $\varepsilon_k = 0.99^k$. In addition, the step size is reinitialized at each iteration before the line--search (i.e., $\alpha_{k+1} = \alpha_0$). The remaining hyperparameters are set as in \cite{Franchini-etal-2024}. We remark that $N_0$ and $\varepsilon_k$ were selected a priori based on the desired increase in the mini--batch size, without empirical validation.

\begin{itemize}
    \item \textbf{Feldkamp--Davis--Kress} ($x_{\text{FDK}}$): it serves as a baseline due to its computational efficiency and widespread use in clinical applications. Notice that, in contrast to $x_{\text{GT}}$, this is computed using undersampled projections (see \textit{Experimental Setup} below).

    \item \textbf{Forward--Backward} ($x_{\text{FB}}$) \cite{Combettes-etal-2005}: a first order deterministic optimization method. The step size is fixed throughout the iterations and set to $\alpha = 10^{-5}$ and there is no line--search.

    \item  \textbf{Variable Metric Inexact Line--search} ($x_{\text{VMILA}}$) \cite{VMILA_or}: a deterministic method that incorporates a variable metric and a non monotone line--search strategy enhanced by adaptive Barzilai--Borwein rules \cite{Barzilai-Borwein-1988,frassoldati}. The initial step size is set to $\alpha_0 = 1$.

    \item \textbf{Adaptive Stochastic Primal--Dual Hybrid Gradient} ($x_{\text{A--SPDHG}}$) \cite{Chambolle-etal-2024}: a\\  stochastic primal--dual gradient method. Unlike FB--LISA, it uses a fixed mini--batch size throughout the iterations, set to $N = 10$, and does not incorporate a batch growth strategy.  The step size is controlled by a parameter $s = 10^{-5}$.
\end{itemize}
\paragraph{Experimental Setup}
We consider three experimental scenarios to evaluate the performance of the compared methods under different acquisition and noise conditions. 

\begin{itemize}
\item Case 1: $n_{\theta}=36$ uniformly spaced  projection angles in $[0,2\pi)$, without additional noise.
\item Case 2: $n_{\theta}=36$ uniformly spaced projection angles in $[0,2\pi)$, with additive Gaussian noise with zero mean and $2\%$ variance.
\item Case 3: $n_{\theta}=72$ uniformly spaced projection angles in $[0,2\pi)$, with additive Gaussian noise with zero mean and $2\%$ variance.
      
\end{itemize}
The cone--beam geometry is chosen to match the specifications of the original walnut scan. The projection operator is implemented using the ASTRA Toolbox \cite{astra}, which ensures efficient and accurate forward and backprojection operations. 
For these three settings, the regularization parameter is set to $\mu = 2, 1, 1$, respectively.

As a stopping criterion, we impose a maximum time budget of 240 seconds for all methods. In addition, intermediate reconstruction quality metrics are recorded at 120 seconds in order to assess the convergence behavior.

\begin{figure}[t]
\captionsetup[subfloat]{labelformat=empty}
	\centering
	\subfloat[ $x_{\text{GT}}$]{
	\scalebox{0.82}{
	\begin{tikzpicture}
	\begin{scope}[spy using outlines={rectangle,red,magnification=2,size=1.5cm}]
	\node [name=c]{	\includegraphics[height=3cm]{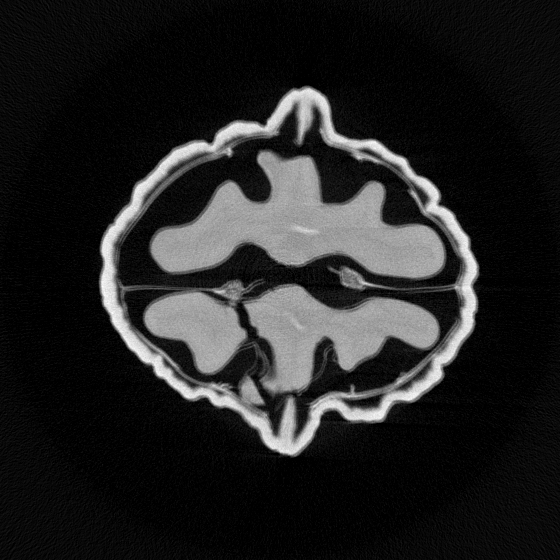}};
	\spy on (0.8,0) in node [name=c1]  at (-0.75,-2.2);
	\spy on (-0.2,-0.5) in node [name=c1]  at ( 0.75,-2.2);
	\end{scope}
	\end{tikzpicture}}}
	\hspace{-0.3cm}
	\subfloat[ \shortstack{$x_{\text{FB}}$ \\ \scriptsize PSNR: 20.73, \\ SSIM: 0.56,\\ \scriptsize HaarPSI: 0.37}]{
	\scalebox{0.82}{
	\begin{tikzpicture}
	\begin{scope}[spy using outlines={rectangle,red,magnification=2,size=1.5cm}]
	\node [name=c]{	\includegraphics[height=3cm]{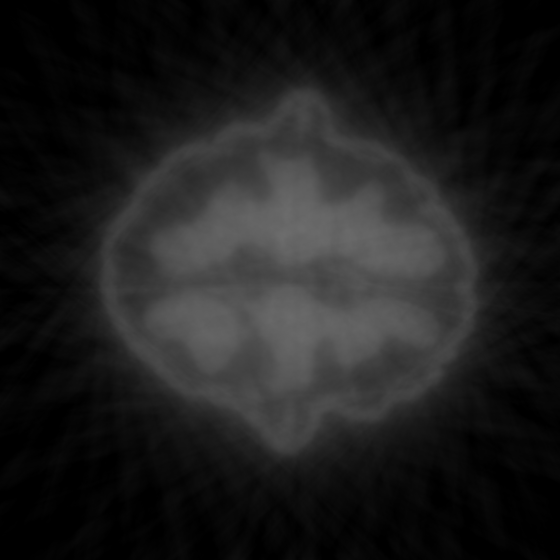}};
	\spy on (0.8,0) in node [name=c1]  at (-0.75,-2.2);
	\spy on (-0.2,-0.5) in node [name=c1]  at ( 0.75,-2.2);
	\end{scope}
	\end{tikzpicture}}}  
	\hspace{-0.3cm}
    \subfloat[ \shortstack{  $x_{\text{VMILA}}$ \\ \scriptsize PSNR: 23.98, \\ SSIM: 0.59,\\ \scriptsize HaarPSI: 0.48}]{
	\scalebox{0.82}{
	\begin{tikzpicture}
	\begin{scope}[spy using outlines={rectangle,red,magnification=2,size=1.5cm}]
	\node [name=c]{	\includegraphics[height=3cm]{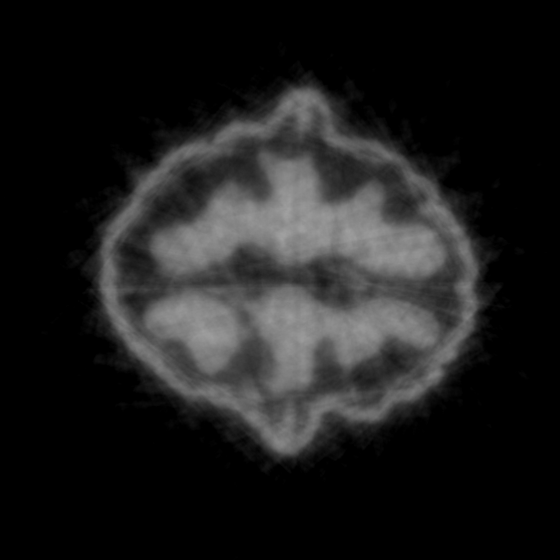}};
	\spy on (0.8,0) in node [name=c1]  at (-0.75,-2.2);
	\spy on (-0.2,-0.5) in node [name=c1]  at ( 0.75,-2.2);
	\end{scope}
	\end{tikzpicture}}}
	 \hspace{-0.3cm}
    \subfloat[  \shortstack{$x_{\text{A--SPDHG}}$ \\ \scriptsize  PSNR: 19.17, \\ SSIM: 0.56,\\ \scriptsize HaarPSI: 0.28}]{
	\scalebox{0.82}{
	\begin{tikzpicture}
	\begin{scope}[spy using outlines={rectangle,red,magnification=2,size=1.5cm}]
	\node [name=c]{	\includegraphics[height=3cm]{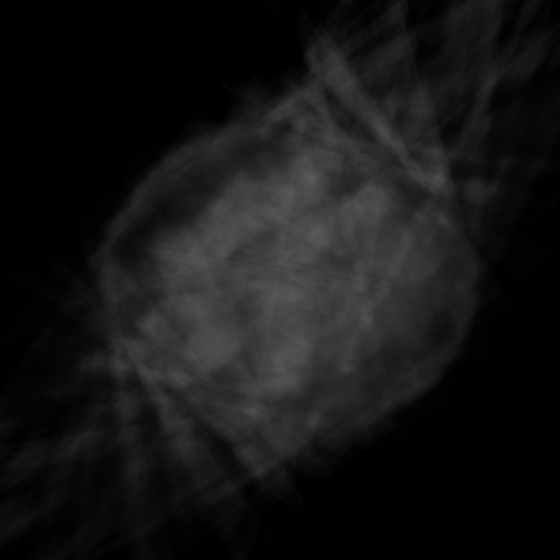}};
	\spy on (0.8,0) in node [name=c1]  at (-0.75,-2.2);
	\spy on (-0.2,-0.5) in node [name=c1]  at ( 0.75,-2.2);
	\end{scope}
	\end{tikzpicture}}} 
    \hspace{-0.3cm}
	\subfloat[ \shortstack{  $x_{\text{FB--LISA}}$ \\ \scriptsize PSNR: 27.53, \\ SSIM: 0.61,\\ \scriptsize HaarPSI: 0.52}]{
	\scalebox{0.82}{
	\begin{tikzpicture}
	\begin{scope}[spy using outlines={rectangle,red,magnification=2,size=1.5cm}]
	\node [name=c]{	\includegraphics[height=3cm]{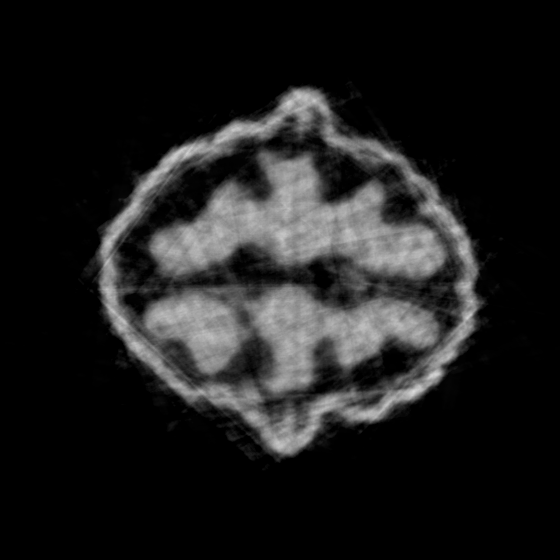}};
	\spy on (0.8,0) in node [name=c1]  at (-0.75,-2.2);
	\spy on (-0.2,-0.5) in node [name=c1]  at ( 0.75,-2.2);
	\end{scope}
	\end{tikzpicture}}}
    \quad\\
    \subfloat[ \shortstack{  $x_{\text{FDK}}$ \\ \scriptsize PSNR: 20.85, \\ SSIM: 0.26,\\ \scriptsize HaarPSI: 0.30}]{
	\scalebox{0.82}{
	\begin{tikzpicture}
	\begin{scope}[spy using outlines={rectangle,red,magnification=2,size=1.5cm}]
	\node [name=c]{	\includegraphics[height=3cm]{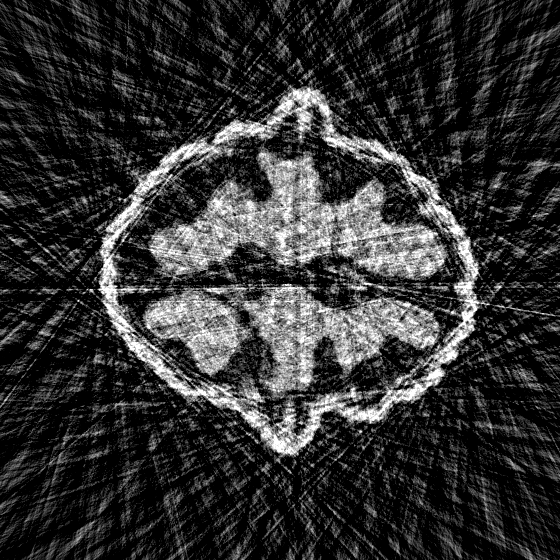}};
	\spy on (0.8,0) in node [name=c1]  at (-0.75,-2.2);
	\spy on (-0.2,-0.5) in node [name=c1]  at ( 0.75,-2.2);
	\end{scope}
	\end{tikzpicture}}}
	\hspace{-0.3cm}
	\subfloat[ \shortstack{$x_{\text{FB}}$ \\ \scriptsize PSNR: 22.15, \\ SSIM: 0.57, \\ \scriptsize HaarPSI: 0.42}]{
	\scalebox{0.82}{
	\begin{tikzpicture}
	\begin{scope}[spy using outlines={rectangle,red,magnification=2,size=1.5cm}]
	\node [name=c]{	\includegraphics[height=3cm]{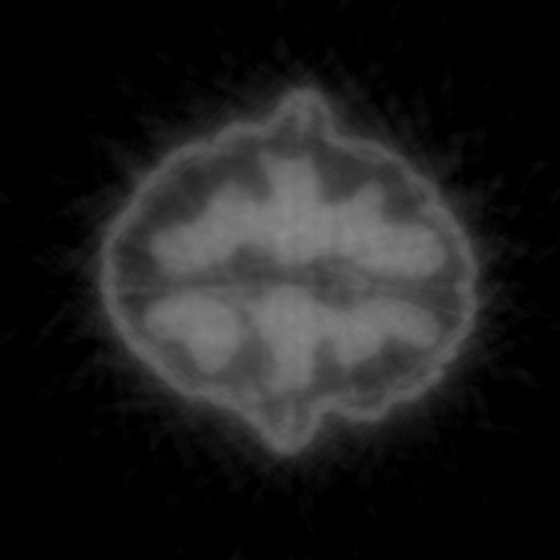}};
	\spy on (0.8,0) in node [name=c1]  at (-0.75,-2.2);
	\spy on (-0.2,-0.5) in node [name=c1]  at (0.75,-2.2);
	\end{scope}
	\end{tikzpicture}}}  
	\hspace{-0.3cm}
    \subfloat[ \shortstack{  $x_{\text{VMILA}}$ \\ \scriptsize PSNR: 26.26, \\ SSIM: 0.61,\\ \scriptsize HaarPSI: 0.51}]{
	\scalebox{0.82}{
	\begin{tikzpicture}
	\begin{scope}[spy using outlines={rectangle,red,magnification=2,size=1.5cm}]
	\node [name=c]{	\includegraphics[height=3cm]{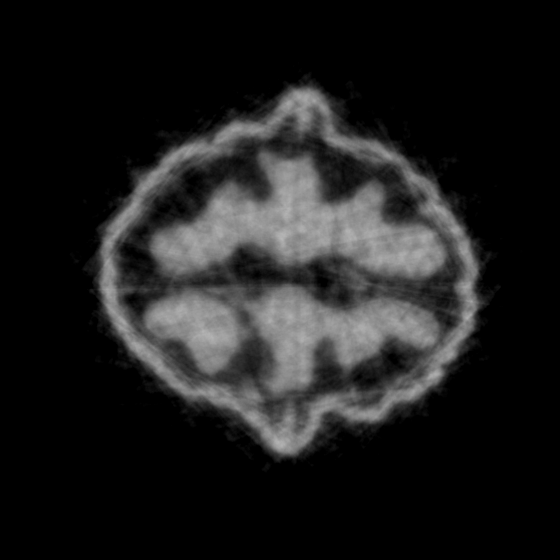}};
	\spy on (0.8,0) in node [name=c1]  at (-0.75,-2.2);
	\spy on (-0.2,-0.5) in node [name=c1]  at (0.75,-2.2);
	\end{scope}
	\end{tikzpicture}}}
	\hspace{-0.3cm}
    \subfloat[ \shortstack{$x_{\text{A--SPDHG}}$ \\ \scriptsize PSNR: 22.93, \\ SSIM: 0.56,\\ \scriptsize HaarPSI: 0.39}]{
	\scalebox{0.82}{
	\begin{tikzpicture}
	\begin{scope}[spy using outlines={rectangle,red,magnification=2,size=1.5cm}]
	\node [name=c]{	\includegraphics[height=3cm]{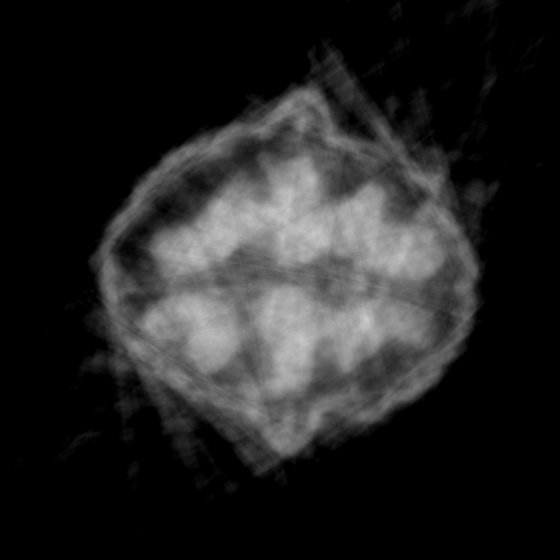}};
	\spy on (0.8,0) in node [name=c1]  at (-0.75,-2.2);
	\spy on (-0.2,-0.5) in node [name=c1]  at (0.75,-2.2);
	\end{scope}
	\end{tikzpicture}}} 
    \hspace{-0.3cm}
	\subfloat[ \shortstack{  $x_{\text{FB--LISA}}$ \\ 
    \scriptsize PSNR: 28.99, \\ SSIM: 0.61,\\ \scriptsize HaarPSI: 0.56}]{
	\scalebox{0.82}{
	\begin{tikzpicture}
	\begin{scope}[spy using outlines={rectangle,red,magnification=2,size=1.5cm}]
	\node [name=c]{	\includegraphics[height=3cm]{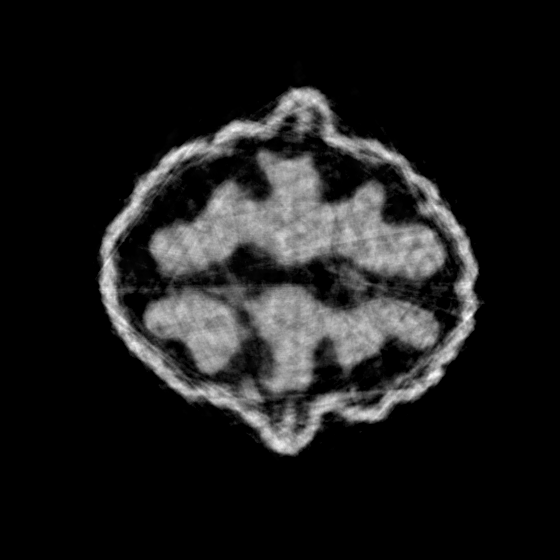}};
	\spy on (0.8,0) in node [name=c1]  at (-0.75,-2.2);
	\spy on (-0.2,-0.5) in node [name=c1]  at (0.75,-2.2);
	\end{scope}
	\end{tikzpicture}}}
    \caption{Reconstruction results for Case 1. The first column shows the ground truth and the FDK reconstruction, while columns 2--4 present the reconstructions of the compared methods (first row: 120 seconds; second row: 240 seconds).}
    \label{fig:recon}
\end{figure}

Fig.~\ref{fig:recon} show the reconstruction corresponding to the axial slice $(:,:,300)$ (with the respective reconstruction quality metrics) for all considered methods at 120 and 240 seconds, respectively, for the scenario described in Case 1. 
The non--iterative reconstruction $x_{\text{FDK}}$ presents many streaking artifacts in the background, due to the sever undersampling of the data. This is a known limitation of FDK  when applied to scarce data. 
For the iterative approaches, at 120 seconds, we observe that, concerning the deterministic methods,  VMILA already provides visually acceptable reconstructions, while $x_{\text{FB}}$ is still far from the reference $x_{\text{GT}}$, with blurred structures and missing details, especially in low--contrast regions. $x_{\text{A--SPDHG}}$ shows limited progress at this stage and remains far from the reference $x_{\text{GT}}$. Instead, $x_{\text{FB--LISA}}$ shows a faster initial progression toward the reference solution in terms of PSNR and SSIM, showing better structural preservation, as indicated by the lowest HaarPSI metric. Notably, the features inside the zoomed--in regions are much more defined for $x_{\text{FB--LISA}}$ compared to all other methods.
At 240 seconds, all iterative methods converge closer to the $x_{\text{GT}}$, with $x_{\text{FB--LISA}}$ reaching the highest PSNR (28.99), getting closest to the reference solution in terms of quantitative metrics, and maintaining competitive SSIM and HaarPSI metrics. 
$x_{\text{FB}}$ and $x_{\text{VMILA}}$ improve over their 120--second performance, but the reconstructed details remain less sharp than those of $x_{\text{FB--LISA}}$. $x_{\text{A--SPDHG}}$ improves but moderately: this is perhaps due to its fixed mini--batch size which appears to limit the reconstruction quality. In contrast,  $x_{\text{FB--LISA}}$ is able to gradually include more and more information in the reconstruction process, thank to
its increasing mini--batch strategy. Again, the zoomed--in elements are more clearly resolved for $x_{\text{FB--LISA}}$ than for the other methods. 

\begin{table}[!ht]
    \centering
    \caption{Quantitative reconstruction metrics for Case 2 }
        \begin{tabular}{ll c c c c}
\toprule
Run time& Methods  & RE & PSNR & SSIM & HaarPSI \\
\midrule
-- & FDK &1.31& 14.25&0.03& 0.24\\
\midrule
\multirow{4}{*}{120 secs} 
& FB&0.63&20.59&0.57&0.37\\
& VMILA&0.41&24.30&0.59&0.46\\
& A--SPDHG&0.62&20.71&0.53&0.27\\
& FB--LISA &\textbf{0.31}&\textbf{26.78}&\textbf{0.61}&\textbf{0.48}\\
\midrule
\multirow{4}{*}{240 secs}
& FB&0.54&21.98&0.58&0.40\\
& VMILA&0.31&26.80&\textbf{0.60}&0.48\\
& A--SPDHG&0.41&24.30&0.58&0.41\\
& FB--LISA &\textbf{0.25}&\textbf{28.69}&\textbf{0.60}&\textbf{0.52}\\
\midrule
\bottomrule
 \end{tabular}
    \label{tab:36_2perc}
\end{table}

Table~\ref{tab:36_2perc} summarizes the reconstruction metrics for all considered methods under noisy conditions for the scenario of Case 2. 
At 120 seconds, FB--LISA already reaches the lowest RE (0.31), indicating a faster approach toward the reference solution and highest PSNR (26.78), outperforming all other methods. FDK, as expected, shows poor performance due to the limited number of views and the noise. A--SPDHG performs better than FB initially in PSNR but still exhibits higher RE and HaarPSI error.
At 240 seconds, all iterative methods improve, with FB--LISA maintaining the best performance across all metrics. VMILA also reaches competitive PSNR and SSIM, while FB improves moderately but remains slightly behind. A--SPDHG benefits from additional time but again  the reconstruction quality remains limited compared to that of FB--LISA.

\begin{table}[!ht]
    \centering
    \caption{Quantitative reconstruction metrics for Case 3 }
        \begin{tabular}{ll c c c c}
\toprule
Run time& Methods  & RE & PSNR & SSIM & HaarPSI \\
\midrule
-- & FDK &0.93& 17.14&0.06& 0.29\\
\midrule
\multirow{4}{*}{120 secs} 
& FB&0.53&22.01&0.61&0.45\\
& VMILA&0.46&23.24&0.59&0.48\\
& A--SPDHG&0.87&17.76&0.62&0.24\\
& FB--LISA &\textbf{0.31}&\textbf{26.88}&\textbf{0.63}&\textbf{0.49}\\
\midrule
\multirow{4}{*}{240 secs}
& FB&0.43&23.87&0.60&0.51\\
& VMILA&0.25&28.54&0.63&0.53\\
& A--SPDHG&0.53&22.07&0.58&0.37\\
& FB--LISA &\textbf{0.22}&\textbf{29.74}&\textbf{0.64}&\textbf{0.55}\\

\midrule
\bottomrule
 \end{tabular}
    \label{tab:72_2perc}
\end{table}
Finally, Table~\ref{tab:72_2perc} reports the  reconstruction metrics for all considered methods for the Case 3.
Also in this case, we draw similar conclusions to those of the previous cases. FDK remains limited due to the sparse--view acquisition, despite the increased number of projection views. Deterministic methods FB and VMILA perform reasonably well, but FB remains slightly behind. A--SPDHG benefits from additional time to sensibly improve the reconstruction metrics, which nevertheless remain below the other methods. Overall, FB--LISA  provides the best reconstruction at already 120 seconds.

\begin{figure}
    \centering
    \includegraphics[width=0.8\linewidth]{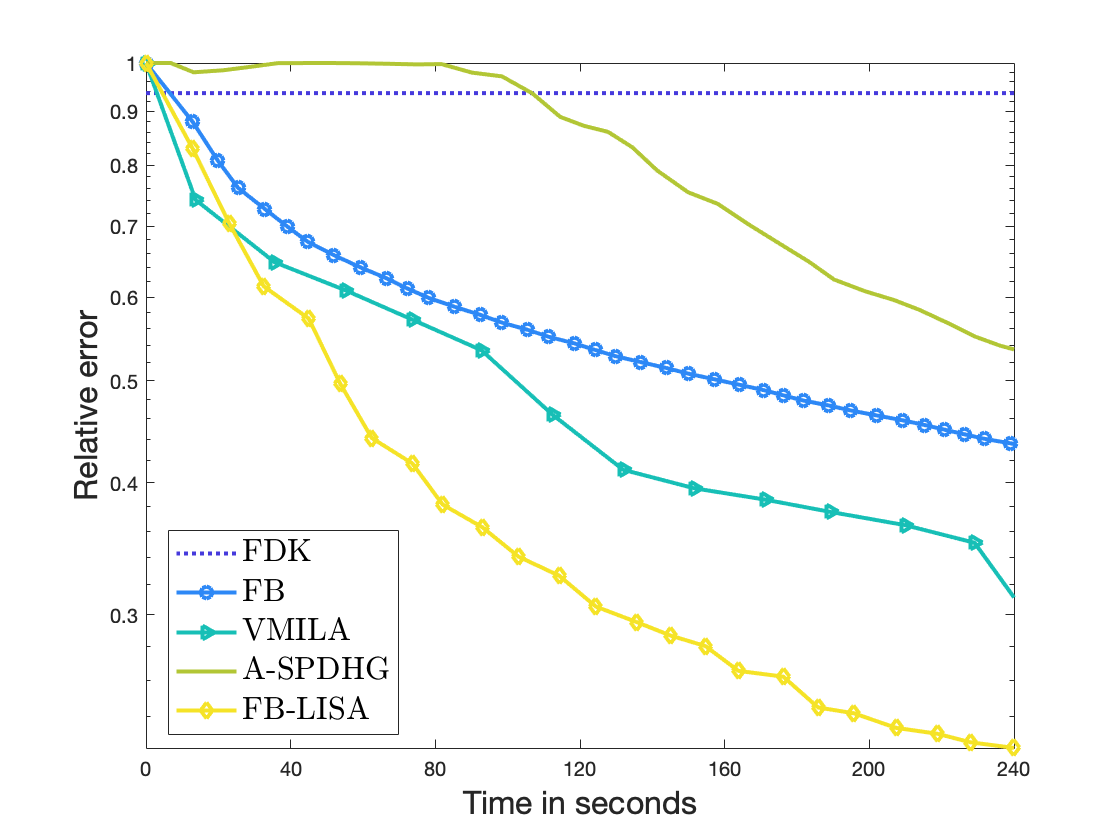}
    \caption{Relative error vs execution time in seconds for Case 3.}
    \label{fig:relerr_72_r2}
\end{figure}

This is further confirmed in Fig.~\ref{fig:relerr_72_r2}, which reports the relative error as a function of execution time (in seconds) for all compared methods. 
FB--LISA (yellow)  consistently shows the fastest decay toward lower error values after 25 seconds, reaching a relative error below $0.25$ at $240\,\text{ seconds}$, approaching the lowest error values among all methods. 
The stochastic nature of FB--LISA is maintained throughout, since the mini--batch size grows from an initial $N_{\text{batch}} =8$ to $N_{\text{batch}} =13$, never reaching the full--batch regime. 
VMILA (teal) exhibits competitive convergence behavior in the early iterations but slows considerably after $150\,\text{ seconds}$, eventually dropping sharply to approximately $0.23$ only at the very end of the run. FB (light blue) steadily progresses toward the solution, although at a slower rate, stagnating around a relative error of $0.43$ at termination. A--SPDHG (olive green) shows little progress in the first $100\,\text{seconds}$ and converges to a final error of about $0.53$, the worst among the iterative methods. Finally, the FDK baseline (blue, dotted) yields a relative error of 0.93, confirming its inadequacy for severely undersampled data.

\section{Conclusions}\label{sec:conclusions}
In this letter, we designed FB--LISA, a first--order stochastic method to handle nonsmooth objectives and solve large--scale, e.g., 3D CT problems. 
FB--LISA relies on adaptive strategies for both learning rate selection (via line--search) and mini--batch size increase.
Notably, the mini--batch growth can be precomputed prior to execution, ensuring compliance with the memory limitations imposed by hardware accelerators. 
Overall, the results across all studied cases indicate that FB–LISA achieves better metrics while converging significantly faster toward the reference solution. 

Future work will focus on the application of FB--LISA to other large--scale imaging inverse problems characterized by computationally expensive forward models, and proving stronger convergence results, such as the almost sure convergence of the whole sequence.

\section*{Acknowledgments}
All authors are members of the Gruppo Nazionale per il Calcolo Scientifico (GNCS) of the Istituto Nazionale di Alta Matematica ``Francesco Severi'' (INdAM), which is kindly acknowledged. This work was partially supported through ``Bando giovani ricercatori 2025 per progetti di ricerca finanziato con contributo 5x1000 anno finanziario 2023'' by the Department of Mathematics and Computer Science of the University of Ferrara.

\appendix

\section{Convergence results}\label{sec:appendix}
This section provides the convergence analysis of the FB--LISA method. The analysis extends \cite[Lemma 2.2, Theorem 2.1]{Franchini-etal-2024}. In particular, we prove the well--definiteness of the line--search and the almost sure stationarity of the limit points.

\begin{lemma}
    Given problem \eqref{eq:gen_prob}, suppose that $f_i:\mathbb{R}^d\rightarrow\mathbb{R}$, $i=1,\dots,N$, are differentiable functions with $L$--Lipschitz continuous gradient. Then the line--search employed by the FB--LISA method is well defined and the steplength $\alpha_k$ satisfies:
    $$
    \tilde{\alpha}=\frac{\beta}{L}< \alpha_k\leq\alpha_{max}.
    $$
\end{lemma}
\begin{proof}
    The Lipschitz continuity of $\nabla f$ allows to apply the descent lemma  \cite[Lemma 5.7]{Beck-2017} and state that
\begin{equation}\label{eq:dis_f1}\tag{A1}
    f_{\mathcal{N}_k}(\bar{x}^{(k)})\leq f_{\mathcal{N}_k}(x^{(k)})+\langle\nabla f_{\mathcal{N}_k}(x^{(k)}),{\bar{x}^{(k)}}-x^{(k)}\rangle+\frac{L}{2}\|\bar{x}^{(k)}-x^{(k)}\|^2.
\end{equation}
As a consequence, the line--search employed by the FB--LISA method for non--smooth composite optimization problem is well defined. Indeed any $\alpha_k\leq\frac{1}{L}$ satisfies inequality (3).\vspace{0.1cm}\\
Let $s$ be the guess value for the line--search at the $k$--th iteration. If $s$ does not satisfy inequality (3), then there exists $\ell_k>0$, such that $\alpha_k=s\beta^{\ell_k}$ satisfies (3), while $\alpha_k/\beta$ does not.  By the descent lemma \cite[Lemma 5.7]{Beck-2017}, this implies that $\alpha_k/\beta > \frac{1}{L}$, and therefore
$$
\alpha_k > \frac{\beta}{L}, \qquad \forall k.
$$
\end{proof}

Next, we prove our main result, Theorem~\ref{thm:MainResult}, whose statement is provided in Section~\ref{ssec:LISA_for_CT}.

\begin{proof}
In view of the Lipschitz continuity of $\nabla f$, we can apply the descent lemma  \cite[Lemma 5.7]{Beck-2017} and state that
\begin{equation}\label{eq:dis_f}\tag{A2}
    f(x^{(k+1)})\leq f(x^{(k)})+\langle\nabla f(x^{(k)}),x^{(k+1)}-x^{(k)}\rangle+\frac{L}{2}\|x^{(k+1)}-x^{(k)}\|^2.
\end{equation}
From the definitions of the proximal operator of $\alpha_k\mathcal{R}$ and the subdifferential of $\mathcal{R}$, it follows that
$$
\frac{1}{\alpha_k}(x^{(k)}-x^{(k+1)}) - \nabla f_{\mathcal{N}_k}(x^{(k)}) \in \partial \mathcal{R}(x^{(k+1)})
$$
and 
\begin{equation}\label{eq:dis_R}\tag{A3}
\mathcal{R}(y)\geq \mathcal{R}(x^{(k+1)})+\langle\frac{1}{\alpha_k}(x^{(k)}-x^{(k+1)})- \nabla f_{\mathcal{N}_k}(x^{(k)}),y-x^{(k+1)}\rangle, \quad \forall y\in\mathbb{R}^d.
\end{equation}
For clarity, we introduce the following notation:
$$
G_k = \frac{1}{\alpha_k}(x^{(k)}-x^{(k+1)});\quad 
e^{(k)} = \nabla f(x^{(k)})-\nabla f_{\mathcal{N}_k}(x^{(k)}).
$$
By combining \eqref{eq:dis_f} and \eqref{eq:dis_R} with $y=x^{(k)}$, we get
\begin{equation}\label{eq:3}\tag{A4}
    \begin{aligned}
        f(x^{(k+1)})+\mathcal{R}(x^{(k+1)})&\leq f(x^{(k)})+\mathcal{R}(x^{(k)}) + \langle e^{(k)}+G_k,x^{(k+1)}-x^{(k)}\rangle +\frac{L}{2}\|x^{(k+1)}-x^{(k)}\|^2\\
        &=f(x^{(k)})+\mathcal{R}(x^{(k)})+ \langle e^{(k)}+G_k,-\alpha_k G_k\rangle + \frac{L}{2}\alpha_k^2\|G_k\|^2\\
        &=f(x^{(k)})+\mathcal{R}(x^{(k)})-\alpha_k\langle e^{(k)},G_k\rangle -\alpha_k \|G_k\|^2+\frac{L}{2}\alpha_k^2\|G_k\|^2\\
        &\leq f(x^{(k)})+\mathcal{R}(x^{(k)}) -\alpha_k \|G_k\|^2+\frac{\alpha_k}{2}\|e^{(k)}\|^2+\frac{\alpha_k}{2}\|G_k\|^2+\\&+\frac{L}{2}\alpha_k^2\|G_k\|^2\\
        &= f(x^{(k)})+\mathcal{R}(x^{(k)})-\frac{\alpha_k}{2}(1-L\alpha_k)\|G_k\|^2+\frac{\alpha_k}{2}\|e^{(k)}\|^2\\
        &\leq f(x^{(k)})+\mathcal{R}(x^{(k)})-\frac{\alpha_k}{4}\|G_k\|^2+\frac{\alpha_k}{2}\|e^{(k)}\|^2,
    \end{aligned}
\end{equation}
where the second inequality follows from the Young's inequality\footnote{Young's inequality: $a^Tb \leq \frac{\|a\|^2}{2\xi} + \frac{\xi\|b\|^2}{2}$, $\forall a,b\in\mathbb{R}^d$, $\xi>0$} and the last inequality holds in view of the assumption on $\alpha_{max}$.\vspace{0.1cm}\\
Now we obtain a bound for $-\|G_k\|^2$ related to  $G^{full}_k {=} \frac{1}{\alpha_k}\left(x^{(k)}{-}\prox_{\alpha_k\mathcal{R}}(x^{(k)}{-}\alpha_k\nabla f(x^{(k)}))\right)$. Particularly,
\begin{equation*}
    \begin{aligned}
        \|G_k^{full}\|^2 &= \|G_k^{full}-G_k+G_k\|^2 = \|G_k^{full}-G_k\|^2+\|G_k\|^2+2\langle G_k,G_k^{full}-G_k\rangle\\
        &\leq \|G_k^{full}-G_k\|^2+\|G_k\|^2+\|G_k\|^2+\|G_k^{full}-G_k\|^2,
    \end{aligned}
\end{equation*}
where the inequality follows by again invoking the Young's inequality. As a consequence,
\begin{equation}\label{eq:G_k_G_k_full}\tag{A5}
\|G_k\|^2\geq \frac{1}{2}\|G_k^{full}\|^2 - \|G_k^{full}-G_k\|^2.
\end{equation}
By plugging \eqref{eq:G_k_G_k_full} into \eqref{eq:3} we can write
\begin{equation*}
\begin{aligned}
    f(x^{(k+1)})+\mathcal{R}(x^{(k+1)})&\leq f(x^{(k)})+\mathcal{R}(x^{(k)})-\frac{\alpha_k}{8}\|G_k^{full}\|^2+\frac{\alpha_k}{4}\|G_k^{full}-G_k\|^2    
    +\frac{\alpha_k}{2}\|e^{(k)}\|^2\\
    &= f(x^{(k)})+\mathcal{R}(x^{(k)})-\frac{\alpha_k}{8}\|G_k^{full}\|^2+\frac{\alpha_k}{2}\|e^{(k)}\|^2+\\
    &\quad +\frac{\alpha_k}{4}\left\|\frac{1}{\alpha_k}\left(x^{(k)}-\prox_{\alpha_k\mathcal{R}}(x^{(k)}-\alpha_k\nabla f(x^{(k)}))\right)+\right.\\
    &\qquad\qquad\left.-\frac{1}{\alpha_k}s\left(x^{(k)}-\prox_{\alpha_k\mathcal{R}}(x^{(k)}-\alpha_k\nabla f_{\mathcal{N}_k}(x^{(k)}))\right)\right\|^2\\
    &\leq f(x^{(k)})+\mathcal{R}(x^{(k)})-\frac{\alpha_k}{8}\|G_k^{full}\|^2+\frac{\alpha_k}{2}\|e^{(k)}\|^2+\frac{\alpha_k}{4}\|e^{(k)}\|^2\\
    &\leq f(x^{(k)})+\mathcal{R}(x^{(k)})-\frac{\tilde{\alpha}}{8}\|G_k^{full}\|^2+\frac{\alpha_{max}}{2}\|e^{(k)}\|^2+\frac{\alpha_{max}}{4}\|e^{(k)}\|^2
    \end{aligned}
\end{equation*}
where the equality follows by the definitions of $G_k^{full}$ and $G_k$, the second inequality is a consequence of the non--expansivity of the proximal operator of $\alpha_k\mathcal{R}$ and the last inequality holds since $\alpha_k\in[\tilde{\alpha},\alpha_{max}]$.
By applying the conditional expected value with respect to the $\sigma$--algebra generated by the iterations $x^{(0)},\dots,x^{(k)}$ to both members of the previous inequality and considering the assumption on $e^{(k)}$, we get 
\begin{equation}\label{eq:5_0} \tag{A6}   \mathbb{E}_k\left[f(x^{(k+1)})+\mathcal{R}(x^{(k+1)})\right]\leq f(x^{(k)})+\mathcal{R}(x^{(k)})-\frac{\tilde{\alpha}}{8}\|G_k^{full}\|^2+\frac{3}{4}\alpha_{max}\varepsilon_k.
\end{equation}

By applying the total expected value, rearranging the terms, summing from $0$ to $K$ and using the telescopic cancellation, it holds that
\begin{equation}\label{eq:5}\tag{A7}
\sum_{k=0}^K \mathbb{E}[\|G_k^{full}\|^2]\leq \frac{8(f(x^{(0)})+\mathcal{R}(x^{(0)}) - F_{inf})}{\tilde{\alpha}}+\frac{3\alpha_{max}\bar{\varepsilon}}{2\tilde{\alpha}},
\end{equation}
where $\bar{\varepsilon}$ exists since the sequence $\{\varepsilon_k\}$ is summable by assumption and hence
$$
\sum_{k=0}^{+\infty}\mathbb{E}[\|G_k^{full}\|^2]< +\infty.
$$
Moreover, by dividing both terms of \eqref{eq:5} by ${K}$, we obtain 
$$
\frac{1}{K}\sum_{k=0}^K \mathbb{E}[\|G_k^{full}\|^2]\leq \frac{8(f(x^{(0)})+\mathcal{R}(x^{(0)}) - F_{inf})}{\tilde{\alpha}K}+\frac{3\alpha_{max}\bar{\varepsilon}}{2\tilde{\alpha}K}.
$$
By subtracting $F_{inf}$ to both members in \eqref{eq:5_0}, the Robbins--Siegmund lemma \cite[Lemma 11, Sec 2.2.2]{Polyak} allows to conclude that $\sum_{k=0}^{\infty}\|G_k^{full}\|^2<+\infty$, a.s. which implies that 
\begin{equation}\label{eq:lim_G_k}\tag{A8}
\lim_{k\rightarrow \infty} \|G_k^{full}\|^2 = \lim_{k\rightarrow \infty}\|x^{(k)}-\prox_{\alpha_k\mathcal{R}}(x^{(k)}-\alpha_k\nabla f(x^{(k)}))\|^2 = 0, \ a.s.
\end{equation}
Let us suppose that there exists a subsequence of $\{x^{(k)}\}$ that converges almost surely to $\bar{x}$, namely there exists $\mathcal{K}\subseteq\mathbb{N}$ such that
$$
\lim_{{k\to\infty, \ k\in\mathcal{K}}} x^{(k)} = \bar{x}
\quad \text{a.s.}
$$
Since $\alpha_k \in [\tilde{\alpha}, \alpha_{\max}]$ for all $k$, the sequence $\{\alpha_k\}_{k\in\mathcal{K}}$ is bounded. Hence, there exist a subsequence $\mathcal{K}' \subseteq \mathcal{K}$ and a scalar $\bar{\alpha} \in [\tilde{\alpha}, \alpha_{\max}]$ such that
$$
\lim_{{k\to\infty, \     k\in\mathcal{K}'}} \alpha_k = \bar{\alpha}.
$$
Moreover, by \eqref{eq:lim_G_k}, it holds that
$$
\lim_{k\to\infty} \left\| x^{(k)} - \prox_{\alpha_k \mathcal{R}}\bigl(x^{(k)} - \alpha_k \nabla f(x^{(k)})\bigr) \right\| = 0
\quad \text{a.s.}
$$
Therefore, along the subsequence $\mathcal{K}'$,
$$
x^{(k)} - \prox_{\alpha_k \mathcal{R}}\bigl(x^{(k)} - \alpha_k \nabla f(x^{(k)})\bigr) \to 0
\quad \text{a.s.}
$$
By continuity of $\nabla f$ and of the proximal operator with respect to both its argument and the parameter $\alpha$, we obtain
$$
\lim_{{k\to\infty, \ k\in\mathcal{K}'}} 
\prox_{\alpha_k \mathcal{R}}\bigl(x^{(k)} - \alpha_k \nabla f(x^{(k)})\bigr)
=
\prox_{\bar{\alpha} \mathcal{R}}\bigl(\bar{x} - \bar{\alpha} \nabla f(\bar{x})\bigr)
\quad \text{a.s.}
$$
Hence,
$$
\bar{x} = \prox_{\bar{\alpha} \mathcal{R}}\bigl(\bar{x} - \bar{\alpha} \nabla f(\bar{x})\bigr)
\quad \text{a.s.}
$$
which is equivalent to
$$
0 \in \nabla f(\bar{x}) + \partial \mathcal{R}(\bar{x}).
$$
Therefore, $\bar{x}$ is a stationary point of problem~(2), almost surely.
\end{proof}

\bibliographystyle{plain}

\bibliography{biblio}

\end{document}